\documentclass[12pt]{amsart}
\usepackage{amsmath, amscd, amsthm, amssymb, graphics, mathrsfs, setspace,fancyhdr,xcolor,geometry,tabularx,shapepar, hyperref}
\usepackage[all]{xy}
\hypersetup{colorlinks=false}
\theoremstyle{plain}
\newtheorem{thm}{Theorem}[section]

\newtheorem{cor}[thm]{Corollary}
\newtheorem{prop}[thm]{Proposition}

\newtheorem*{thm*}{Theorem}
\newtheorem*{problem*}{Problem}
\newtheorem*{cor*}{Corollary}

\theoremstyle{definition}
\newtheorem{defn}[thm]{Definition}

\newtheorem{rem}[thm]{Remark}

\newtheorem{ex}[thm]{Example}

\numberwithin{equation}{section}

\newcommand{\Spec}{\operatorname{Spec}}

\newcommand{\Z}{{\mathbb Z}}

\newcommand{\pone}{{\mathbb P}^1}

\newcommand{\SB}{\operatorname{SB}}

\newcommand{\gm}{{{\mathbb G}_{m}}}

\newcommand{\gr}{\operatorname{Gr}}
\newcommand{\ind}{\operatorname{ind}}

\begin{document}

\title[$K_1$-zero-cycles on certain symplectic involution varieties]{Zero-cycles with coefficients for the second generalized symplectic involution variety of an algebra of degree 4}

\author{Patrick K. McFaddin}
\address{Department of Mathematics, University of South Carolina, 
Columbia, SC 29208}
\email{pkmcfaddin@gmail.com}
\urladdr{\url{http://mcfaddin.github.io/}}
\keywords{Algebraic cycles; $K$-cohomology; Severi-Brauer varieties; Central simple algebras; Algebraic groups; Homogeneous varieties}

\subjclass[2010]{Primary \href{http://www.ams.org/msc/msc2010.html?t=16K20&s=14c25&btn=Search&ls=Ct}{14C25}, Secondary \href{http://www.ams.org/msc/msc2010.html?t=&s=16k20&btn=Search&ls=s}{16K20}}

\maketitle
\addtocounter{section}{0}

\begin{abstract}
We compute the group of $K_1$-zero-cycles on the second generalized involution variety for an algebra of degree 4 with symplectic involution.  This description is given in terms of the group of multipliers of similitudes associated to the algebra with involution.  Our method utilizes the framework of Chernousov and Merkurjev for computing $K_1$-zero-cycles in terms of $R$-equivalence classes of prescribed algebraic groups. This gives a computation of $K_1$-zero-cycles for some homogeneous varieties of type $\mathsf{C}_2$.
\end{abstract}

\section{Introduction}

The formalism of ``Chow groups with coefficients," as developed by Markus Rost \cite{Rost}, has become a useful tool for investigating intersection-theoretic questions in arithmetic and algebraic geometry. This theory has bolstered the study of algebraic cycles by providing a broader context and more desirable functoriality properties than those found for classical Chow groups.  Given a variety $X$ defined over a field $F$, there is a complex \cite{Quillen, Rost}$$\cdots \to \coprod _{x \in X_{(p+1)}} K_{p+q+1}(F(x)) \to \coprod _{x \in X_{(p)}} K_{p+q}(F(x)) \to \coprod_{x \in X_{(p-1)}} K_{p+q-1} (F(x)) \to \cdots $$ where $X_{(i)}$ is the set of points of $X$ of dimension $i$, $K_i$ are Quillen $K$-groups, and the differentials are given component-wise by residue and corestriction homomorphisms.  The $p^{\text{th}}$ homology group of this complex, i.e., homology at the middle term above, is denoted $A_p(X, K_q)$, and classical Chow groups in the sense of \cite{Fulton} are recovered via the identification $\text{CH}_p(X) = A_p(X, K_{-p})$.  The study of these $K$-cohomology groups for homogenous varieties has seen many useful applications to Galois cohomology, and significant results include the Merkurjev-Suslin Theorem \cite{MS82} and the Milnor and Bloch-Kato Conjectures \cite{Voe}.  Unfortunately, general descriptions of these groups remain elusive, and computations are done in various cases, described below.

Algebraic cycles have also found a great deal of utility in the study of algebraic groups.  In \cite{ChernMerkSpin}, Chernousov and Merkurjev show that for certain groups $G$ and appropriately chosen $G$-homogenous varieties $X$, the group $A_0(X, K_1)$ encodes arithmetic properties of $G$ that can be used to investigate rationality by means of $R$-equivalence.  Such results allow for the utilization of algebraic $K$-theory and $K$-cohomology in the analysis of such questions.  In the present paper, we focus on the reverse implication. Our main result does not shed light on the rationality problem for the symplectic groups (shown to be rational in \cite{Weil, ChernMerkSU}), but crucially utilizes the Chernousov-Merkurjev framework to identify $K$-cohomology groups.

Computing groups of zero-cycles (with coefficients) on homogeneous varieties has been taken up by a number of authors.  The groups of $K_0$-zero-cycles, i.e., Chow groups $\text{CH}_0(X)$, were computed for e.g.  some (generalized) Severi-Brauer varieties, (generalized) involution varieties, Severi-Brauer flag varieties, and varieties of exceptional type in \cite{ChernMerk, Karpenko, KrashZeroCycles, Mer95, PaninThesis, PSZ, Swan}.


For $K_1$-zero-cycles, many cases remain unknown. Let $G$ be a simply connected group of inner type $\mathsf{A}_n$, so that $G  = \textbf{SL}_1(A)$ for a central simple algebra $A$ over $F$.  For any $1 \leq k \leq \text{deg}(A)$, the (generalized) Severi-Brauer varieties $X= \SB_k(A)$ associated to $A$ are homogeneous for $G$ and their $K_1$-zero-cycles have been computed in few cases.  For $X$ the Severi-Brauer variety of $A$ (i.e., $k=1$),  Merkurjev and Suslin have shown that $A_0(X, K_1) \simeq K_1(A)$ \cite{MerkSus}. For $X$ the second generalized Severi-Brauer variety (i.e., $k=2$) of an algebra $A$ of index 4, it was shown by the author that $A_0(X, K_1) \simeq \{(a, \lambda) \in K_1(A) \times F^{\times} \mid \text{Nrd}_A (a) = \lambda ^2\}$ \cite{McFaddin}.

For simply connected groups of types $\mathsf{B}_n$ and $\mathsf{D}_n$, Chernousov and Merkurjev provide a description of the $K_1$-zero-cycles on the involution varieties associated to quadratic spaces and their twisted forms.  Indeed, it was shown in \cite{ChernMerkSpin} that for $X$ the projective quadric defined by a quadratic form $q$ on a vector space $V$, there is a canonical isomorphism $A_0(X, K_1) \simeq \Gamma ^{+}(V, q)/ R\text{Spin}(V, q)$.  If $X$ is the involution variety corresponding to an algebra $A$ of index no greater than 2 with quadratic pair $(\sigma, f)$, there is an isomorphism $A_0(X, K_1) \simeq \Gamma ^{+}(A, \sigma, f) /R \text{Spin}(A, \sigma, f)$.  Note that this work also recovers a theorem of Rost showing that $\overline{A}_0(X, K_1)=0$ for $X$ the projective quadric of a Pfister form \cite{Rost88}, where $\overline{A}_0(X, K_1)$ denotes the kernel of the pushforward $A_0(X, K_1) \to K_1(F)$ along the structure morphism $X \to \Spec F$. This fact was used in the proof of the Bloch-Kato Conjecture \cite{Voe}.

In the present paper, we treat type $\mathsf{C}_2$.  Such groups are given by symplectic groups $\textbf{Sp}(A, \sigma)$ for a central simple algebra of degree $4$ with symplectic involution.  The second generalized involution variety $\text{IV}_2(A, \sigma)$ is homogeneous for $\textbf{Sp}(A, \sigma)$ and we compute its group of $K_1$-zero-cycles.  Our main result is the following: 

\begin{thm*}[\ref{thm:mainthm}]
Let $(A, \sigma)$ be a central simple algebra over $F$ of degree 4 with symplectic involution and let $X = \operatorname{IV}_2(A, \sigma)$ be the second generalized involution variety associated to $(A, \sigma)$.  Then $A_0(X, K_1)$ is given by the group of multipliers of similitudes of $(A, \sigma)$, i.e., there is a group isomorphism $$A_0(X, K_1) \simeq G(A, \sigma):= \{\sigma(a) \cdot a \mid a \in A^{\times} \operatorname{ satisfies } \sigma(a) \cdot a \in F^{\times}\}.$$ Moreover, $\overline{A}_0(X, K_1) =0$.
\end{thm*}

In the case where $A$ is split, i.e., $A \simeq \text{End}_F(V)$ for an $F$-vector space $V$, the involution $\sigma$ is adjoint to an alternating bilinear form $b$ on $V$ and $\text{IV}_2(A, \sigma)$ is the second isotropic (or symplectic) Grassmannian $\text{SGr}(2, V)$ consisting of planes in $V$ which are $b$-isotropic. In this case, we obtain the following description of the group of $K_1$-zero-cycles:

\begin{cor*}[\ref{cor:split}]
Let $(A, \sigma)$ be a central simple algebra of degree 4 with symplectic involution and let $X = \operatorname{IV}_2(A, \sigma)$.  If $A$ is split, then $A_0(X, K_1) \simeq F^{\times}$.
\end{cor*}

\subsection*{Organization}

In Section~\ref{section:prelim} we define a slew of homogeneous varieties and recall some structure results.  We then define cycle modules and $K$-cohomology groups.  In Section~\ref{section:alggrp} we present the main algebraic groups of interest: spin, Clifford, (general) symplectic.  We recall some exceptional identifications of groups of type $\mathsf{B}_2$ and $\mathsf{C}_2$ and their algebraic invariants.  We define $R$-triviality and the index of an algebraic group, and present the framework of Chernousov and Merkurjev describing the group of $K_1$-zero-cycles in terms of $R$-equivalence classes.  In Section~\ref{section:Requiv}, we present our main result for symplectic involution varieties of algebras of degree 4.

\section{Preliminaries}\label{section:prelim}
Throughout, $F$ denotes an infinite perfect field of arbitrary characteristic.  If $L/F$ is a field extension and $A$ is an $F$-algebra, we write $A_L:= A \otimes _F L$.  By a \emph{scheme} we will mean a separated scheme of finite type over the field $F$.  A variety is an integral scheme.  If $X$ and $Y$ are $F$-varieties, we write $X_Y : = X\times _{\Spec F} Y$.

A \emph{central simple algebra} over $F$ is a finite-dimensional $F$-algebra with no two-sided ideals other than $(0)$ and $(1)$ and whose center is precisely $F$.  Recall that the dimension of a central simple algebra $A$ is a square, and we define the \emph{degree} of $A$ to be $\text{deg}(A) = \sqrt{\dim A}$.  One may write $A = M_n(D)$ for a division algebra $D$, unique up to isomorphism, and we define the \emph{index} of $A$ to be $\text{ind}(A) = \text{deg}(D).$  The \emph{period} (or exponent) of an algebra $A$ is its order in the Brauer group of $F$.   


An \emph{involution} $\sigma$ on a central simple algebra $A$ is an anti-automorphism satisfying $\sigma \circ \sigma = \text{id}_A$.  An involution is \emph{of the first kind} if $\sigma |_{F} = \text{id}_F$.  Upon extending scalars to a separable closure, involutions become adjoint to bilinear forms.  Involutions which are adjoint to an alternating bilinear form (respectively, symmetric bilinear form) over a separable closure of $F$ are \emph{symplectic} (respectively, \emph{orthogonal}).  Note that an algebra with symplectic involution necessarily has even degree \cite[Prop 2.6]{MerkBook}.  For an algebra with involution of the first kind $(A, \sigma)$, the set of \emph{symmetrized elements} is $$\text{Symd}(A, \sigma) = \{a + \sigma (a) \mid a\in A\}.$$

\subsection{Severi-Brauer and Involution Varieties}

We refer the reader to \cite{Blanch, MerkBook, KrashZeroCycles, Tao} for a full treatment of the following material.  Throughout this section, let $A$ be a central simple algebra of degree $n$.  Involutions on $A$ are assumed to be of the first kind, either orthogonal or symplectic.

\begin{defn}
For a (right or left) ideal $I \subset A$, the \emph{reduced dimension} of $I$ is $$\text{rdim}(I) = \frac{\text{dim}_F I}{n}.$$
\end{defn}

\begin{defn} For any integer $1 \leq k \leq n$, the $k^{\text{th}}$ \emph{generalized Severi-Brauer variety} $\SB_k(A)$ of $A$ is the variety of right ideals of dimension $nk$ (or reduced dimension $k$) in $A$.   
\end{defn}

Such a varietiy is a twisted form of the Grassmannian $\gr(k, n)$ of $k$-dimensional subspaces of an $n$-dimensional vector space.  The variety $\SB_1(A) = \SB(A)$ is the usual Severi-Brauer variety of $A$, which is a twisted form of projective space $\mathbb{P}^{n-1} = \gr(1, n)$.  For a field extension $L/F$, the variety $\SB_k(A)$ has an $L$-rational point if and only if $\text{ind}(A_L) \mid k$ \cite[Prop. 1.17]{MerkBook}.

\begin{defn}  Let $(A, \sigma)$ be a central simple algebra with involution $\sigma$ (of the first kind, either orthogonal or symplectic).  A right ideal $I$ is $\sigma$-\emph{isotropic} if $\sigma(I) \cdot I = (0)$.  For any integer $1 \leq k \leq n$, the $k^{\text{th}}$ \emph{generalized involution variety} $\text{IV}_k(A, \sigma)$ is the variety of $\sigma$-isotropic right ideals of reduced dimension $k$ \cite[$\S$8]{KrashZeroCycles}.  
\end{defn}

Forgetting the isotropy condition defines an embedding $\text{IV}_k(A, \sigma) \hookrightarrow \SB _k(A).$  In the case where $k = 1$ and $\sigma$ is orthogonal, the above embedding realizes $\text{IV}_1(A, \sigma)$ as a twisted projective quadric.  This variety was first defined and studied in \cite{Tao}.

For algebras of degree 4, exceptional identifications allow us to recognize additional structure via certain exceptional identifications (see Section~\ref{section:excid}).  A geometric consequence is given by the following result of Krashen.

\begin{prop}[\cite{KrashZeroCycles}, Prop. 8.10]\label{prop:quadhyper}
Suppose $(A, \sigma)$ is a degree 4 central simple algebra with symplectic involution over a field $F$ with $\operatorname{char}(F) \neq 2$.  Then $\operatorname{IV}_2(A, \sigma)$ is isomorphic to a quadric hypersurface in $\mathbb{P}^4$.
\end{prop} 

We explicitly identify this quadric hypersurface in terms of $\sigma$-symmetrized elements in $A$ below (see Remark~\ref{rem:quadhyp}).

\subsection{Cycle Modules and Homology}

Cycle modules were introduced by M. Rost in \cite{Rost} and good references are \cite{EKM, GMS}.
\begin{defn} 
A \emph{cycle module} $M$ over $F$ is a function assigning to every field extension $L/F$ a graded abelian group $M(L) = M_*(L)$, which is a graded module over the Milnor $K$-theory ring $K^M_*(F)$ satisfying certain data and compatibility axioms.  This data includes
\begin{enumerate}
\item For each field homomorphism $L \to E$ over $F$, there is a degree 0 homomorphism $r_{E/L}: M(L) \to M(E)$ called \emph{restriction}.
\item For each field homomorphism $L \to E$ over $F$, there is a degree 0 homomorphism $c_{E/L}: M(E) \to M(L)$ called \emph{corestriction} (or norm).
\item For each extension $L/F$ and each (rank 1) discrete valuation $v$ on $L$, there is a degree $-1$ homomorphism $\partial _{v}: M(L) \to M(\kappa(v))$ called the \emph{residue homomorphism}, where $\kappa(v)$ is the residue field of $v$.
\end{enumerate}
These homomorphisms are compatible with the corresponding maps in Milnor $K$-theory (see D1-D4, R1a-R3e of \cite[Def. 1.1]{Rost}).  For an $F$-scheme $X$, we say that a cycle module $M$ over $F$ is a cycle module on $X$ if it further satisfies properties FD and C of \cite[Def. 2.1]{Rost} concerning those extension of $F$ which arise as residue fields of points of $X$.
\end{defn}

Let $X$ be an $F$-variety and $M$ a cycle module on $X$.  Utilizing this axiomatic framework and setting $$C_p(X, M_q) = \coprod _{x \in X_{(p)}} M_{p+q}(F(x)),$$ one obtains a complex $$ \cdots \xrightarrow{} C_{p+1}(X, M_{q}) \xrightarrow{d_X} C_p(X, M_q) \xrightarrow{d_X} C_{p-1}(X, M_{q}) \xrightarrow{} \cdots$$  with differentials $d_X$ induced by the homomorphisms $\partial _v$ associated to discrete valuations of subvarieties of codimension 1.  This is often referred to as the \emph{Rost complex} \cite{EKM} or \emph{Gersten complex} \cite{Panin}.  We denote the homology group at the middle term by $A_p(X, M_q)$.  Our main focus will be the group $$A_0(X, K_1) = \operatorname{coker} \left( \coprod _{x\in X_{(1)}} K_2(F(x))\xrightarrow{d_X} \coprod _{x \in X_{(0)} } K_1(F(x)) \right)$$ of $K_1$-zero-cycles, where $K_*$ is the cycle module given by Quillen or Milnor $K$-theory, as these coincide in low degree.  References for Quillen $K$-theory include \cite{Srinivas, Kbook} and for Milnor $K$-theory see \cite{GS06}.  The group of $K_1$-zero-cycles is therefore given by equivalences classes of expressions $\sum (\alpha, x)$, where $x$ is a closed point on $X$ and $\alpha$ and element of $F(x)^{\times} = K_1(F(x))$.  Note that the group $A_0(X, K_1)$ coincides with the cohomology group $H^d(X, \mathcal{K}_{d+1})$, where $d = \text{dim}(X)$ and $\mathcal{K}_i$ is the Zariski sheaf associated to the Quillen $K$-theory presheaf $U \mapsto K_i(U)$.  We also note that $A_0(X, K_1)$ coincides with the motivic cohomolgy group $H^{2d+1}(X, \Z(d+1))$ \cite[Lem. 4.11]{Voe2}

\begin{defn}\label{def:reduced}
Let $X$ be a complete $F$-variety.  For $p  \geq 0$ define the \emph{norm homomorphism} $N^p: A_0(X, K_p) \to K_p(F)$ by $$N^p\left(\sum (\alpha_x, x)\right) = \sum c_{F(x)/F}(\alpha_x),$$ where $c_{F(x)/F}$ is the corestriction.  The kernel of $N^p$ will be denoted $\overline{A}_0(X, K_p)$.  In low degree (e.g., for $p = 0, 1$) these norm homomorphisms coincide with those defined in terms of Milnor $K$-theory \cite[\S 1.5]{ChernMerkSpin}.  For the remainder, we denote the norm map $N^0$ by $\text{deg}: \text{CH}_0(X) \to \Z$, noting that $A_0(X, K_0) = \text{CH}_0(X)$ and $K_0(F) = \Z$, and refer to it as the \emph{degree homomorphism} \cite{ChernMerk}.  Its image is equal to $\ind(X) \Z$.
\end{defn}

\section{Algebraic Groups}\label{section:alggrp}

Let $A$ be a central simple algebra over $F$, and define the \emph{general linear group} $\textbf{GL}_1(A)$ of invertible elements in $A$ by $\textbf{GL}_1(A)(R) = A^{\times}_R$ for any commutative $F$-algebra $R$.  We write $\gm = \textbf{GL}_1(F)$.  We will be interested in certain subgroups of $\textbf{GL}_1(A)$ which reflect information about involutions on $A$.

Let $(A, \sigma)$ be an algebra with symplectic involution.  For any commutative $F$-algebra $R$, let $$\textbf{GSp}(A, \sigma)(R) = \{ a \in A_R ^{\times} \mid a \cdot \sigma_R(a) \in R^{\times}\}$$ $$\textbf{Sp}(A,\sigma)(R) = \{a \in A^{\times}_R \mid a \cdot \sigma_R (a) = 1\}.$$ These are (non-respectively) the \emph{symplectic} and \emph{general symplectic} groups of $(A, \sigma)$.  In the above definition, $\sigma_R: A_R \to A_R$ is the involution obtained by extending scalars to $R$.  Notice that we have a group homomorphism $ \mu : \textbf{GSp}(A, \sigma) \to \gm$ defined by $\mu_R = a \cdot \sigma_R(a)$ on $R$-points.  We call this the \emph{multiplier map} (associated to $\sigma$) which clearly has kernel $\textbf{Sp}(A, \sigma)$.

Now let $(V, q)$ be a non-degenerate quadratic space and define the \emph{Clifford algebra} $C(V, q)$ to be the quotient of the tensor algebra $T(V) = \bigoplus_i V^{\otimes i}$ by the ideal $I(q)$ generated by all elements of the form $v \otimes v - q(v) \cdot 1$ with $v \in V$. The $\mathbb{N}$-grading on $T(V)$ induces a $\Z/2$-grading on $C(V, q) = C_0(V, q) \oplus C_1(V, q),$ and we call $C_0(V,q)$ the \emph{even Clifford algebra}.  Define the \emph{even Clifford group} by $$\mathbf{\Gamma}^+(V, q)(R) = \{ g \in C_0(V, q)^{\times}_R \mid g \cdot V_R \cdot g^{-1} = V_R\}.$$  The identity map on $V$ extends to an involution on $T(V)$ which fixes the ideal $I(q)$.  The induced involution on $C(V, q)$ restricts to $C_0(V, q)$ and we denote the resulting involution by $\tau$, and call it the \emph{canonical involution} on $C_0(V, q)$.  The multiplier map associated to $\tau$ defines a homomorphism $\text{Sn}: \mathbf{\Gamma}^+(V, q) \to \gm$ via $x \mapsto x \cdot \tau (x)$, called the \emph{spinor norm} and its kernel is denoted $$\textbf{Spin}(V, q) = \{ g \in \mathbf{\Gamma}^+(V, q)(R)  \mid \text{ }  g \cdot \tau(g) = 1\}.$$

\begin{rem}
We follow the same notational convention as \cite{ChernMerkSpin}, utilizing $\text{GSp}(A, \sigma)$, $\text{Sp}(A, \sigma)$, $\Gamma^+(V, q)$, and $\text{Spin}(V, q)$ to denote the groups of $F$-points of the corresponding algebraic groups $\textbf{GSp}(A, \sigma), \textbf{Sp}(A, \sigma), \mathbf{\Gamma}^+(V, q)$, and $\textbf{Spin}(V, q)$.  
\end{rem}

Let $(A, \sigma)$ be an algebra with symplectic involution.  The multiplier map associated to $\sigma$ gives rise to a short exact sequence \cite[$\S$23.A]{MerkBook}$$ 1 \to \textbf{Sp}(A, \sigma) \to \textbf{GSp}(A, \sigma) \xrightarrow{\mu} \gm \to 1.$$  Evaluating on $F$-points, this yields an exact sequence of groups \cite[Prop. 22.10]{MerkBook} $$1 \to \text{Sp}(A, \sigma) \to \text{GSp}(A,\sigma) \xrightarrow{\mu} F^{\times}.$$  The image of the multiplier map (on $F$-points) is called the \emph{group of multipliers of similitudes of} $(A, \sigma)$, denoted $G(A, \sigma)  \subset F^{\times}$ \cite[$\S$12.B]{MerkBook}.  This gives a short exact sequence $1 \to \text{Sp}(A, \sigma) \to \text{GSp}(A,\sigma) \xrightarrow{\mu} G(A, \sigma) \to 1$.

Similarly, for a quadratic space $(V, q)$, the spinor norm fits into a short exact sequence $$ 1 \to \textbf{Spin}(V, q) \to \mathbf{\Gamma}^{+} (V, q) \xrightarrow{\text{Sn}} \gm \to 1,$$  and evaluating on $F$-points gives an exact sequence $1 \to \text{Spin}(V,q) \to \Gamma^+(V, q) \to F^{\times}$.  The image of the spinor norm (on $F$-points) is called the \emph{group of spinor norms on} $(V, q)$, denoted $\text{Sn}(V, q)$ \cite[Def. 13.30]{MerkBook}.

\subsection{Exceptional Identifications}\label{section:excid}

Here we recall some facts given in \cite[$\S$15.C]{MerkBook}.  Let $\mathsf{B}_2$ be the groupoid of \emph{oriented} quadratic spaces \cite[Def. 12.40]{MerkBook} of dimension 5, where the morphisms are isometries which preserve the orientation.  Let $\mathsf{C}_2$ be the groupoid of central simple algebras of degree 4 with symplectic involution, where the morphisms are the $F$-algebra isomorphisms which preserve the involutions.

Define $\textbf{C}: \mathsf{B}_2 \to \mathsf{C}_2$ via $(V, q, \zeta) \mapsto (C_0(V, q), \tau)$, where $\tau$ is the canonical involution on the even Clifford algebra $C_0(V, q)$.

The map in the other direction requires further development.  Let $(A, \sigma)$ be a central simple algebra of degree 4 with symplectic involution and consider the collection $\text{Symd}(A, \sigma)$ of $\sigma$-symmetrized elements in $A$.  The reduced characteristic polynomial of an element $a \in \text{Symd}(A, \sigma)$ is a square and takes the form $$\text{Prd}_{A, a}(x) = \left(x^2 - \text{Trp}_{\sigma}(a)x + \text{Nrp}_{\sigma}(a)\right)^2.$$  Here, $\text{Trp}_{\sigma}$ is the \emph{Pfaffian trace} and is a linear form on $\text{Symd}(A, \sigma)$;  $\text{Nrp}_{\sigma}$ is the \emph{Pfaffian norm} and is a quadratic form on $\text{Symd}(A, \sigma)$.  Notice that if $\text{Trp}_{\sigma}(a) = 0$, then $ 0 = a^2 + \text{Nrp}_{\sigma}(a)$ since $a$ satisfies its own characteristic polynomial.

Let $\text{Symd}(A, \sigma)^0 = \{ a \in \text{Symd}(A, \sigma) \mid \text{Trp}_{\sigma}(a) = 0 \}$.  Then $- \text{Nrp}_{\sigma}(a) = a^2$ defines a quadratic form on $\text{Symd}(A, \sigma)^0$ which we will denote by $s_{\sigma}$.  We have thus associated a quadratic space (of dimension 5) to $(A, \sigma)$.  Lastly, there is a unique orientation $\eta$ on this quadratic space, and we refer the reader to loc. cit.  Define $\textbf{S}: \mathsf{C}_2 \to \mathsf{B}_2$ via $(A, \sigma) \mapsto (\text{Symd}(A, \sigma)^0, s_{\sigma}, \eta)$.

\begin{prop}[\cite{MerkBook}, Prop. 15.16]
The functors $\operatorname{\mathbf{C}} : \mathsf{B}_2 \to \mathsf{C}_2$ and $\operatorname{\mathbf{S}} : \mathsf{C}_2 \to \mathsf{B}_2$ define an equivalence of groupoids.
\end{prop}

\begin{prop}[\cite{MerkBook}, Prop. 15.18]\label{prop:inv}
Suppose that $(A, \sigma) \in \mathsf{C}_2$ and $(V, q, \zeta) \in \mathsf{B}_2$ correspond to one another under the above groupoid equivalence.  Then the special Clifford group and general symplectic group coincide: $\Gamma^+(V, q) = \operatorname{GSp}(A, \sigma)$.  The spin group and symplectic group coincide: $\operatorname{Spin}(V, q) = \operatorname{Sp}(A, \sigma)$.  The group of spinor norms and the group of multipliers coincide: $\operatorname{Sn}(V, q) = G(A, \sigma)$.
\end{prop}

The following proposition is in some sense an algebraic restatement of of Proposition~\ref{prop:quadhyper}, although there is no requirement that $\text{char}(F) \neq 2$.

\begin{prop}[\cite{MerkBook}, Prop. 15.20]\label{prop:projquad}
(1).  For every right ideal $I \subset A$ of reduced dimension 2, the intersection $I_{\sigma} := I \cap \operatorname{Symd}(A, \sigma)$ is a 1-dimensional subspace of $\operatorname{Symd}(A, \sigma)$ which is isotropic for the quadratic form $\operatorname{Nrp}_{\sigma}$.  If $\sigma(I) \cdot I = 0$ (i.e. if $I$ is $\sigma$-isotropic), then $I_{\sigma} \subset \operatorname{Symd}(A, \sigma)^0$ and is thus $s_{\sigma}$-isotropic.

(2).  For every nonzero vector $x \in \operatorname{Symd}(A, \sigma)$ with $\operatorname{Nrp}_{\sigma}(x) = 0$, the right ideal $xA$ has reduced dimension 2.  Moreover, if $x \in \operatorname{Symd}(A, \sigma)^0$ then $xA$ is $\sigma$-isotropic.
\end{prop}

\begin{rem}\label{rem:quadhyp}
This proposition gives mutually inverse maps $I \mapsto I \cap \text{Symd}(A, \sigma)$ and $xF \to xA$ between the sets of right ideals of $A$ of reduced dimension 2 and $\text{Nrp}_{\sigma}$-isotropic subspaces of $\text{Symd}(A, \sigma)$.  Geometrically, this defines an isomorphism $$\SB _2(A) \overset{\sim}{\longrightarrow} \{\text{Nrp}_{\sigma} = 0\} \subset \mathbb{P}(\text{Symd}(A, \sigma)) \cong \mathbb{P}^5.$$  Restricting to the collection of $\sigma$-isotropic ideals, one obtains a bijection with the collection of 1-dimensional $s_{\sigma}$-isotropic subspaces of $\text{Symd}(A, \sigma) ^0$.  Hence, we have an isomorphism $$\text{IV}_2(A, \sigma) \overset{\sim}{\longrightarrow} \{s_{\sigma} = 0\} \subset \mathbb{P}(\text{Symd}(A, \sigma) ^0) \cong \mathbb{P}^4.$$  This realizes $\text{IV}_2(A, \sigma)$ as a quadric hypersurface in $\mathbb{P}^4$ as stated in Proposition~\ref{prop:quadhyper} above.  That is, $\text{IV}_2(A, \sigma)$ is obtained from $\SB_2(A)$ by intersecting with the hyperplane $   \{\text{Trd}_A = 0\}$.
\end{rem}

\subsection{$R$-equivalence and $K_1$-zero-cycles} Here we define the notions of $R$-triviality and index of algebraic groups which we use to describe the Chernousov-Merkurjev framework for computing $K_1$-zero-cycles.

\begin{defn}  See \cite{ChernMerkSpin}. Let $G$ be an algebraic group over $F$.  A point $x \in G(F)$ is called $R$-\emph{trivial} if there is a rational morphism $f: \pone \dashrightarrow G$, defined at 0 and 1,  and with $f(0) = 1$ and $f(1) = x$.  The collection of all $R$-trivial elements of $G(F)$ is denoted $RG(F)$ and is a normal subgroup of $G(F)$.  We denote the set of $R$-equivalence classes $G(F)/RG(F)$ by $G(F)/R$.  If $H$ is a normal closed subgroup of $G$ then $RH(F)$ is a normal subgroup of $G(F)$ \cite[Lemma 1.2]{ChernMerkSU}.

\begin{defn}
Let $G$ be an algebraic group and $\rho: G \to \gm$ a character on $G$.  The \emph{index} of $\rho$, denoted $\text{ind}(\rho)$, is the least positive integer in the image of the composition $$G(F((t))) \xrightarrow{\rho} F((t))^{\times} \xrightarrow{v} \Z.$$  Here $F((t))$ is the field of formal Laurent series, $v$ is the discrete valuation on $F((t))$ defined by taking the smallest power of $t$, and $\rho$ denotes the induced character $\rho_{F((t))}$.
\end{defn}

\end{defn}

For the proofs of our main results, we utilize the framework established in \cite{ChernMerkSpin} which recovers groups of $K_1$-zero-cycles of homogeneous varieties in terms of algebraic groups and corresponding $R$-equivalence classes.  Let $G$ be an algebraic group over $F$.  Assume $\rho: G \to \gm$ is a nontrivial character of $G$ with kernel $H$, and $X$ is a smooth complete $F$-variety satisfying
\begin{enumerate}
\item[(I)]  $G$ is connected, reductive, and rational and $H$ is smooth.
\item[(II)] For any field extension $L/F$, the degree homomorphism $\text{deg}: \text{CH}_0(X_L) \to \Z$ is injective.
\item[(III)] For any field extension $L/F$ satisfying $X(L) \neq \emptyset$ the norm homomorphism $N^1_L: A_0(X_L, K_1) \to L^{\times}$ is an isomorphism.
\item[(IV)] For any field extension $L/F$, $\text{ind}(\rho_L) = \text{ind}(X_L)$.
\end{enumerate}

\begin{thm}[\cite{ChernMerkSpin}]\label{thm:chernmerk}
For $(G, \rho, X)$ as above, i.e., satisfying $\operatorname{(I)-(IV)}$, there are canonical isomorphisms $$G(F)/ RH(F) \simeq A_0(X, K_1) $$ $$H(F)/R \simeq \overline{A}_0(X, K_1).$$
\end{thm}

\begin{ex}[Quadratic Spaces \cite{ChernMerkSpin}]\label{ex:K1ZCIV}
Let $(V, q)$ be a non-degenerate quadratic space over $F$ of dimension $n \geq 2$.  The triple $\left(\mathbf{\Gamma}^+(V, q), \text{Sn}, \text{IV}(V, q)\right)$ consisting of the Clifford group, the spinor norm, and the involution variety associated to $(V, q)$  satisfies conditions (I)-(IV).  Hence there are canonical isomorphisms $$ A_0(X, K_1) \simeq \Gamma^+(V, q)/R\operatorname{Spin}(V, q), $$ $$\overline{A}_0(X, K_1) \simeq \operatorname{Spin}(V, q)/R.$$
\end{ex}

As discussed above, quadratic spaces of dimension 5 correspond to algebras of degree 4 with symplectic involution under the exceptional identification of $\mathsf{B}_2$ (which coincides with $ \mathsf{B}_2'$ for $\text{char}(F) \neq 2)$ and $\mathsf{C}_2$.  In the following section, we will use this fact to transfer the isomorphisms given above to the case of algebras with symplectic involution.

\section{Main Result}\label{section:Requiv}

In light of Example~\ref{ex:K1ZCIV}, we state our main result for algebras of degree 4. Recall that for an algebra with symplectic involution $(A, \sigma)$, the group $G(A, \sigma)$ of multipliers of similitudes is given by $G(A, \sigma) = \{\sigma(a) \cdot a \mid a \in \text{GSp}(A, \sigma)\}.$

\begin{thm}\label{thm:mainthm}
Let $(A, \sigma)$ be an algebra of degree 4 with symplectic involution and let $X = \operatorname{IV}_2(A, \sigma)$.  The group of $K_1$-zero-cycles on $X$ is the group of multipliers of similitudes of $(A, \sigma)$, i.e., there is a  canonical group isomorphism $$A_0(X, K_1) \simeq \operatorname{GSp}(A, \sigma)/\operatorname{Sp}(A, \sigma) = G(A, \sigma) $$  Moreover, $\overline{A}_0(X, K_1) = 0$.
\end{thm}

\begin{proof}
Let $V = \text{Symd}(A, \sigma)^0$.  The variety $\text{IV}_2(A, \sigma)$ is isomorphic to the quadric hypersurface $Q = \{s_{\sigma} = 0\} \subset \mathbb{P}(V)$ by Proposition~\ref{prop:quadhyper} and Proposition~\ref{prop:projquad}.  Thus, it suffices to compute $K_1$-zero-cycles on $Q$.  Taking $G = \mathbf{\Gamma}^+(V, s_{\sigma})$ and $\rho$ the spinor norm, Example~\ref{ex:K1ZCIV} yields canonical isomorphisms $$A_0(Q, K_1) \simeq \Gamma^+(V, s_{\sigma})/ R\text{Spin}(V, s_{\sigma}),$$ $$\overline{A}_0(Q, K_1) \simeq \text{Spin}(V, s_{\sigma})/R.$$ Since $(A, \sigma)$ corresponds to $(V, s_{\sigma})$ under the group of equivalence of Proposition~\ref{prop:inv}, we may write these quotients in terms of symplectic groups using the identifications $\Gamma^+(V, s_{\sigma}) = \text{GSp}(A, \sigma)$ and $\text{Spin}(V, s_{\sigma}) = \text{Sp}(A, \sigma)$.  Since $\textbf{Sp}(A, \sigma)$ is rational \cite[Prop. 2.4]{ChernMerkSU}, it is $R$-trivial and we have $R \text{Sp}(A, \sigma) = \text{Sp}(A, \sigma)$ \cite{KTAG}.  Thus, $A_0(X, K_1) = \text{GSp}(A, \sigma)/\text{Sp}(A, \sigma)$ and $\overline{A}_0(X, K_1) = 0$.  Lastly, evaluating the short exact sequence of algebraic groups $1 \to \textbf{Sp}(A, \sigma) \to \textbf{GSp}(A, \sigma) \to \gm \to 1$ on $F$-points gives $A_0(X, K_1) \simeq \operatorname{GSp}(A, \sigma)/\operatorname{Sp}(A, \sigma) =G(A, \sigma)$.
\end{proof}

\begin{cor}\label{cor:split}
Let $(A, \sigma)$ and $X$ be as above.  If $A$ is split, then $A_0(X, K_1) \simeq F^{\times}$.
\end{cor}

\begin{proof}
This follows from $G(A, \sigma) \simeq F^{\times}$ in the split case \cite[Prop. 12.20]{MerkBook}.
\end{proof}

\begin{rem}
Let $(A, \sigma)$ be an algebra of degree 4 with symplectic involution.  The map $X = \text{IV}_2(A, \sigma) \to \SB_2(A) =Y$, given by forgetting isotropy, induces a map $A_0(X, K_1) \to A_0(Y, K_1)$ which can be described explicitly.  We use the identification $A_0(X, K_1) \simeq G(A, \sigma)$ given above together with $$A_0(Y, K_1) \simeq\{(a, \lambda) \in K_1(A) \times F^{\times} \mid \text{Nrd}_A (a) = \lambda ^2\},$$ established in \cite{McFaddin}.  By \cite[Prop. 12.23]{MerkBook}, for any $g \in \text{GSp}(A, \sigma)$, we have the relation $\text{Nrd}_A(g) = \mu(g) ^2$. The assignment $\mu(g) \mapsto (\mu(g), g)$ gives the desired map. 
\end{rem}

\subsection*{Acknowledgements}  We would like to thank Danny Krashen for helpful comments and suggestions.  We would also like to thank the quick referee for their useful comments and for spotting an error in an earlier draft.

\end{document}